\documentclass[12pt]{article}
\usepackage{amsmath,latexsym}
\usepackage{amsthm}
\usepackage{enumerate}
\usepackage{epsfig}

\newcommand\NN{\mathrm{I\!N}}
\newcommand\RR{\mathrm{I\!R}}

\newcommand\eps{{\varepsilon}}

\newcommand{\thankyou}{Supported in part by Japanese GCOE Program G08: ``Fostering Top Leaders in Mathematics - Broadening the Core and Exploring New Ground".}

\newcommand{\address}{Address: Department of Mathematics, University of North Texas, 1155 Union Circle \#311430, Denton, TX 76203-5017, USA; E-mail: allaart@unt.edu, kiko@unt.edu}

\newtheorem{theorem}{Theorem}[section]
\newtheorem{definition}[theorem]{Definition}
\newtheorem{corollary}[theorem]{Corollary}
\newtheorem{example}[theorem]{Example}
\newtheorem{lemma}[theorem]{Lemma}
\newtheorem{remark}[theorem]{Remark}

\title{The improper infinite derivatives of \\ Takagi's nowhere-differentiable function} 

\author{Pieter C. Allaart and Kiko Kawamura \\
University of North Texas\footnote{\thankyou}\ \,\footnote{\address}}
\date{\today}

\begin{document}
\pagestyle{myheadings}

\maketitle

\begin{abstract}

Let $T$ be Takagi's continuous but nowhere-differentiable function. 
Using a representation in terms of Rademacher series due to N.~Kono [{\em Acta Math. Hungar.} {\bf 49} (1987)], we give a complete characterization of those points where $T$ has a left-sided, right-sided, or two-sided infinite derivative. This characterization is illustrated by several examples. A consequence of the main result is that the sets of points where $T'(x)=\pm\infty$ have Hausdorff dimension one. As a byproduct of the method of proof, some exact results concerning the modulus of continuity of $T$ are also obtained.



\end{abstract}

\section{Introduction}

Takagi's function is one of the simplest examples of a nowhere-differentiable continuous function. It was first discovered in 1903~\cite{Takagi}, and is defined by 
\begin{equation*}
T(x)=\sum_{n=1}^\infty \frac{1}{2^n} \phi^{(n)}(x), \qquad 0\leq x\leq 1,
\label{eq:takagi-function}
\end{equation*}
where $\phi^{(1)}:=\phi$ is the ``tent map" defined by
\begin{equation*}
\phi(x):=\begin{cases}
2x, & \mbox{if $0\leq x\leq 1/2$},\\
2-2x, & \mbox{if $1/2\leq x\leq 1$};
\end{cases}
\end{equation*}
and inductively, $\phi^{(n)}:=\phi\circ\phi^{(n-1)}$ for $n\geq 2$. 


Takagi's function was rediscovered independently by Van der Waerden, Hildebrandt, De Rham and others, and is known alternatively as Van der Waerden's function. Although $T$ does not have a finite derivative anywhere, it is known to have an improper infinite derivative at many points. At which points exactly this is the case appeared to be settled in 1936 by Begle and Ayres \cite{Begle}. Let $O_n$ be the number of zeros, and $I_n=n-O_n$ the number of ones, among the first $n$ binary digits of $x$, and let $D_n=O_n-I_n$. Begle and Ayres claimed that $T'(x)=\infty$ if $D_n\to\infty$, and $T'(x)=-\infty$ if $D_n\to-\infty$. Unfortunately, in their proof they considered only the case $D_n\to\infty$, and only the right-hand derivative, believing the condition for the left-hand derivative to be the same. It is not. In fact, Kruppel \cite{Kruppel}, unaware of Begle and Ayres' paper, recently published a counterexample to their claim, which we explain in Section \ref{sec:example} below.

The main purpose of the present article, then, is to give a complete characterization of those points $x$ at which $T$ has an improper infinite derivative. Guided by Kruppel's counterexample, we replace the condition of Begle and Ayres by a stronger condition, expressed in terms of the binary expansion of $x$. Since the condition we obtain is somewhat intransparent, we illustrate it with several examples. This is done in Section \ref{sec:main-result}. The main result is proved in Section \ref{sec:proof}, using a representation in terms of Rademacher series due to Kono \cite{Kono}. In Section \ref{sec:modulus} we extend, with little extra effort, another recent result of Kruppel \cite{Kruppel} concerning the modulus of continuity of $T$.

\section{Kruppel's counterexample} \label{sec:example}

The following example, which is essentially Example 7.2 of  \cite{Kruppel}, shows that $T'(x)$ may not exist even if $D_n\to\infty$. We present the argument here in a somewhat different (and, we hope, easier to visualize) form. This section may be skipped without loss of continuity. It does, however, lay down the basic idea upon which the proof of necessity in Section \ref{sec:proof} is based.

Let $x=\sum_{n=1}^\infty 2^{-a_n}$, where $a_n=4^n$. For this $x$, we certainly have $D_n\to\infty$. A well-known formula for $T(x)$ at dyadic rational points is
\begin{equation}
T\left(\frac{k}{2^m}\right)=\frac{1}{2^m}\sum_{j=0}^{k-1}(m-2s_j),
\label{eq:Takagi-dyadic}
\end{equation}
where $s_j$ is the number of ones in the binary representation of the integer $j$. (See, e.g., \cite{Kruppel}, p.~44.) For given $m$, let $k$ be the integer such that $k/2^m<x<(k+1)/2^m$. Then
\begin{equation*}
T\left(\frac{k+1}{2^m}\right)-T\left(\frac{k}{2^m}\right)=\frac{1}{2^m}(m-2s_k)=\frac{1}{2^m} D_m,
\end{equation*}
so the secant slopes over the dyadic intervals $[k/2^m,(k+1)/2^m]$ containing $x$ indeed tend to $+\infty$. However, if put $m=a_{n+1}-1$, then $s_k=n$ whereas $s_{k-1}=n+a_{n+1}-a_n-2$ and $s_{k-2}=n+a_{n+1}-a_n-3$. Thus, a simple calculation using \eqref{eq:Takagi-dyadic} yields
\begin{align*}
2^m\left[T\left(\frac{k+1}{2^m}\right)-T\left(\frac{k-2}{2^m}\right)\right]
&=3m-2s_k-2s_{k-1}-2s_{k-2}\\
&=4a_n-a_{n+1}-6n+7\to-\infty,
\end{align*}
as $n\to\infty$. Since the intervals $[(k-2)/2^m,(k+1)/2^m]$ also contain $x$, it follows that $T$ cannot have an infinite derivative at $x$. 

It is easy to imagine how this idea can be extended for sequences $\{a_n\}$ which do not grow quite as fast as $4^n$, by enlarging the intervals even further to the left; that is, we can take the secant slopes over $[(k-j)/2^m,(k+1)/2^m]$ where $j=3,4,\dots$. In fact, we can even let $j$ depend on $m$. It is essentially this realization that lead us to the correct condition for the existence of an improper derivative at a point $x$, as stated in the next section. But, since we wish to consider the left-hand and right-hand derivatives separately, we will use a slightly different approach that does not make use of \eqref{eq:Takagi-dyadic}.

\section{Improper derivatives} \label{sec:main-result}

Define
\begin{align*}
T'_{+}(x)&:=\lim_{h \downarrow 0} \frac{T(x+h)-T(x)}{h}, \\
T'_{-}(x)&:=\lim_{h \uparrow 0} \frac{T(x+h)-T(x)}{h},
\end{align*}
provided each limit exists as an extended real number. It has been pointed out by various authors (e.g. \cite{Begle,Kruppel}) that if $x$ is a dyadic rational (that is, a point of the form $x=k/2^m$), then $T'_+(x)=+\infty$ and $T'_-(x)=-\infty$. We now treat the non-dyadic case.

\begin{theorem} \label{thm:main}
Let $x\in(0,1)$ be non-dyadic, and write
\begin{equation}
x=\sum_{n=1}^\infty 2^{-a_n}, \qquad 1-x=\sum_{n=1}^\infty 2^{-b_n},
\label{eq:expansions}
\end{equation}
where $\{a_n\}$ and $\{b_n\}$ are strictly increasing sequences of positive integers, determined uniquely by $x$. Then:\vspace{-2mm}
\begin{enumerate}[(i)]\setlength{\itemsep}{-1mm}
\item $T'_+(x)=+\infty$ if and only if $a_n-2n\to\infty$.
\item $T'_-(x)=+\infty$ if and only if
\begin{equation}
a_{n+1}-2a_n+2n-\log_2(a_{n+1}-a_n)\to-\infty.
\label{eq:NS-condition+}
\end{equation}
\item $T'_+(x)=-\infty$ if and only if 
\begin{equation}
b_{n+1}-2b_n+2n-\log_2(b_{n+1}-b_n)\to-\infty.
\label{eq:NS-condition-}
\end{equation}
\item $T'_-(x)=-\infty$ if and only if $b_n-2n\to\infty$.
\end{enumerate}
\end{theorem}

\begin{corollary}
In the notation of Theorem \ref{thm:main}, we have:\vspace{-2mm}
\begin{enumerate}[(i)]\setlength{\itemsep}{-1mm}
\item $T'(x)=+\infty$ if and only if \eqref{eq:NS-condition+} holds;
\item $T'(x)=-\infty$ if and only if \eqref{eq:NS-condition-} holds.
\end{enumerate}
\end{corollary}

\begin{proof}
The condition \eqref{eq:NS-condition+} implies that $a_n-2n\to\infty$. For:
\begin{align*}
a_{n+1}-2a_n+2n-&\log_2(a_{n+1}-a_n)\\
&=a_{n+1}-a_n-\log_2(a_{n+1}-a_n)-(a_n-2n)\\
&\geq -(a_n-2n).
\end{align*}
This gives the first statement; the second follows by symmetry.
\end{proof}

\begin{remark}
{\rm
The condition $a_n-2n\to\infty$ is equivalent to the condition of Begle and Ayres. First, if $D_k\to\infty$, then $a_n-2n=D_{a_n}\to\infty$. Conversely, suppose $a_n-2n\to\infty$. Then for $a_n\leq k<a_{n+1}$, 
\begin{equation*}
D_k=k-2I_k=k-2n\geq a_n-2n\to\infty.
\end{equation*}
Conditions \eqref{eq:NS-condition+} and \eqref{eq:NS-condition-}, on the other hand, may look a bit mysterious. The examples below aim to provide more insight into their meaning. Since the conditions are quite analogous, we focus on \eqref{eq:NS-condition+}.
}
\end{remark}

\begin{example}
{\rm
If the number of consecutive zeros in the binary expansion of $x$ is bounded, say by $M$, then $a_{n+1}-a_n\leq M+1$, and so
\begin{equation*}
a_{n+1}-2a_n+2n-\log_2(a_{n+1}-a_n)\leq M+1-(a_n-2n).
\end{equation*}
Thus, we obtain Kruppel's result \cite[Proposition 5.3]{Kruppel}: if $D_n\to\infty$ and the number of consecutive $0$'s in the binary expansion of $x$ is bounded, then $T'(x)=+\infty$. Similarly, if $D_n\to-\infty$ and the number of consecutive $1$'s is bounded, then $T'(x)=-\infty$.
}
\end{example}

\begin{example}
{\rm
If $\limsup_{n\to\infty}a_{n+1}/a_n>2$, then \eqref{eq:NS-condition+} fails. To see this, write $a_{n+1}=\lambda_n a_n$. Then, whenever $\lambda_n\geq 2+\eps$,
\begin{align*}
a_{n+1}-2a_n+2n-&\log_2(a_{n+1}-a_n)\\
&=(\lambda_n-2)a_n+2n-\log_2((\lambda_n-1)a_n)\\
&\geq(\lambda_n-2)a_n+2n-\log_2((\lambda_n-2)a_n)-(2/\eps)\\
&\geq 2n-(2/\eps)\to\infty,
\end{align*}
where the first inequality follows since, by the mean value theorem,
\begin{equation*}
\log_2(\lambda_n-1)-\log_2(\lambda_n-2)\leq\frac{1}{(\lambda_n-2)\log 2}<2/\eps.
\end{equation*}
Thus, even if the $4$ in Kruppel's counterexample in Section \ref{sec:example} is replaced by a smaller number $\gamma>2$, $T$ will not have an improper derivative at $x$.
}
\end{example}

\begin{example}
{\rm
On the other hand, a sufficient condition for \eqref{eq:NS-condition+} to hold is that, for some $0<\eps\leq 1$,
\begin{equation}
\limsup_{n\to\infty}\frac{a_{n+1}}{a_n}=2-\eps\qquad\mbox{and}\qquad \liminf_{n\to\infty}\frac{a_n}{n}>\frac{2}{\eps}.
\label{eq:sufficient}
\end{equation}
(We leave the easy verification to the reader.) Thus, for instance, \eqref{eq:NS-condition+} holds for $a_n=3n$; for any increasing polynomial of degree $2$ or higher; and for any exponential sequence $a_n=\lfloor\alpha^n\rfloor$ with $1<\alpha<2$. As another example, let $a_n$ be the $n$-th prime number; then $a_n/n\log n\to 1$ by the Prime Number Theorem. Thus, $a_n$ satisfies \eqref{eq:sufficient} with $\eps=1$, and hence it satisfies \eqref{eq:NS-condition+}.
}
\end{example}

If $\limsup_{n\to\infty}a_{n+1}/a_n=2$, then a finer examination of the asymptotics of the sequence $\{a_n\}$ is necessary, as the next example shows.

\begin{example}
{\rm
The sequence $a_n=2^n$ does not satisfy \eqref{eq:NS-condition+}; neither does $a_n=2^n+n$. But $a_n=2^n+(1+\eps)n$ satisfies \eqref{eq:NS-condition+} for any $\eps>0$:
\begin{align*}
a_{n+1}-2a_n+2n-\log_2(a_{n+1}-a_n)
&=(1-\eps)n+1+\eps-\log_2(2^n+1+\eps)\\
&\leq -\eps n+1+\eps\to-\infty.
\end{align*}
}
\end{example}

This last example also illustrates that the logarithmic term in \eqref{eq:NS-condition+} can sometimes be of critical importance.

\bigskip
An important subset of $[0,1]$ is formed by the points $x$ whose binary expansion has a density; that is, points $x=\sum_{k=1}^\infty 2^{-k}\eps_k$ for which the limit
\begin{equation}
d_1(x):=\lim_{n\to\infty}\frac{1}{n}\sum_{k=1}^{n}\eps_k
\label{eq:density1}
\end{equation}
exists. Note that $d_1(x)$ expresses the long-run proportion of $1$'s in the binary expansion of $x$. If it exists, we define
\begin{equation}
d_0(x):=1-d_1(x)
\label{eq:density0}
\end{equation}
to denote the long-run proportion of $0$'s. An immediate consequence of Theorem \ref{thm:main} is that $T(x)$ has an infinite derivative at the majority of points $x$ for which $d_1(x)$ exists.

\begin{corollary} \label{cor:density}
Let $x$ be a non-dyadic point and suppose $d_1(x)$ exists. If either \vspace{-1mm}
\begin{enumerate}[(a)]\setlength{\itemsep}{0mm}
\item $0<d_1(x)<1/2$, or
\item $d_1(x)=0$ and $\limsup_{n\to\infty}a_{n+1}/a_n<2$,
\end{enumerate}
then $T'(x)=+\infty$. Similarly, if either \vspace{-1mm}
\begin{enumerate}[(a)]\setlength{\itemsep}{0mm}
\item $1/2<d_1(x)<1$, or
\item $d_1(x)=1$ and $\limsup_{n\to\infty}b_{n+1}/b_n<2$,
\end{enumerate}
then $T'(x)=-\infty$.
\end{corollary}

\begin{proof}
By the definition \eqref{eq:density1}, $n/a_n\to d_1(x)$ and $n/b_n\to d_1(1-x)=d_0(x)$. In particular, if $0<d_1(x)<1$, it follows that $a_{n+1}/a_n\to 1$ and $b_{n+1}/b_n\to 1$. Thus, under the conditions of the corollary, \eqref{eq:sufficient} (or its analog for the sequence $\{b_n\}$) is satisfied.
\end{proof}

Corollary \ref{cor:density} has a remarkable consequence for the Hausdorff dimension of the set of points where $T'(x)=\pm\infty$. (See \cite{Falconer} for the definition and basic properties of Hausdorff dimension.)

\begin{corollary}
Let $S_\infty=\{x\in[0,1]:T'(x)=\infty\}$, and $S_{-\infty}=\{x\in[0,1]:T'(x)=-\infty\}$. Then
\begin{equation*}
\dim_H S_\infty=\dim_H S_{-\infty}=1,
\end{equation*}
where $\dim_H$ denotes the Hausdorff dimension.
\end{corollary}

\begin{proof}
By Corollary \ref{cor:density}, $S_\infty$ contains the sets
$F(\alpha):=\{x\in[0,1]:d_1(x)=\alpha\}$, for $0<\alpha<1/2$.
It is well known that
\begin{equation*}
\dim_H F(\alpha)=\frac{-\alpha\log(\alpha)-(1-\alpha)\log(1-\alpha)}{\log2}
\end{equation*}
(see \cite[Proposition 10.1]{Falconer}). Thus,
\begin{equation*}
\dim_H S_\infty \geq \dim_H \bigcup_{0<\alpha < 1/2}F(\alpha)
=\sup_{0<\alpha < 1/2} \dim_H F(\alpha) = 1.
\end{equation*}
The dimension of $S_{-\infty}$ follows in the same way.
\end{proof}

Corollary \ref{cor:density} left out the (binary) {\em normal numbers}; that is, those numbers $x$ for which $d_1(x)=1/2$. These numbers form a set of Lebesgue measure one by Borel's theorem. However, the law of the iterated logarithm implies that for almost all of those, $\limsup_{n\to\infty}D_n=+\infty$ and $\liminf_{n\to\infty}D_n=-\infty$. Hence, at almost all normal numbers, $T$ does not even have a one-sided infinite derivative. Less extremely, for any {\em rational} normal number $x$ (such as $x=1/3$), $D_n$ oscillates between finite bounds so that $T'(x)$ does not exist. 

Nonetheless, many normal numbers satisfy $T'(x)=\pm\infty$. For instance, if $a_n=2n+\lfloor \sqrt{n}\rfloor$ or $a_n=2n+\lfloor\log n\rfloor$, it is easy to see that \eqref{eq:NS-condition+} is satisfied. On the other hand, surprisingly perhaps, there exist normal numbers for which $T_+'(x)=+\infty$, but $T_-'(x)$ fails to exist. Here we construct one such example.

\begin{example}
{\rm
Let $a_1=3$, and for $n\geq 1$, define $a_{n+1}$ recursively as follows. If $a_n\leq 2n+\lfloor\sqrt{n}\rfloor$, put $a_{n+1}=2n+3\lfloor\sqrt{n}\rfloor$; otherwise, put $a_{n+1}=a_n+1$. Since $a_n$ always increases by at least $1$ and $2n+\lfloor\sqrt{n}\rfloor$ increases by at most $3$ at each step, it is clear that for every $n$, $a_n\geq 2n+\lfloor\sqrt{n}\rfloor-1$. Hence $a_n-2n\to\infty$. Furthermore, $a_n\leq 2n+3\lfloor\sqrt{n}\rfloor$ for each $n$, and so $a_n/n\to 2$. Finally, it is easy to check that $a_n\leq 2n+\lfloor\sqrt{n}\rfloor$ for infinitely many $n$. Thus, infinitely often,
\begin{align*}
a_{n+1}-2a_n+2n-&\log_2(a_{n+1}-a_n)\\
&\geq 2n+3\lfloor\sqrt{n}\rfloor-2(2n+\lfloor\sqrt{n}\rfloor)+2n-\log_2(2\lfloor\sqrt{n}\rfloor+1)\\
&\geq\lfloor\sqrt{n}\rfloor-\frac12\log n-2\to\infty.
\end{align*}
}
\end{example}

\section{Proof of the main theorem} \label{sec:proof}

To prove Theorem \ref{thm:main} we will use an approach by N. Kono \cite{Kono}. Let $x$ and $h$ be real numbers such that $0\leq x<x+h<1$, and write
\begin{equation*}
x=\sum_{k=1}^\infty 2^{-k}\eps_k, \qquad x+h=\sum_{k=1}^\infty 2^{-k}\eps_k',
\end{equation*}
where $\eps_k,\eps_k'\in\{0,1\}$. When $x$ is dyadic rational, there are two binary expansions, but we choose the one which is eventually all zeros. 

For $h>0$, let $p:=p(h) \in \NN$ such that $2^{-p-1}<h\leq 2^{-p}$ and let
\begin{equation*}
k_0:=\max\{k: \eps_1=\eps_1',\dots,\eps_k=\eps_k'\}
\end{equation*}
(or $k_0=0$ if $\eps_1\neq \eps_1'$). Clearly $0\leq k_0\leq p$, and we have the implications
\begin{align}
k_0<p \quad&\Rightarrow\quad \eps_{k_0+1}=0 \mbox{ and } \eps_{k_0+1}'=1, \label{eq:kono-fact1}\\
k_0+2\leq p \quad&\Rightarrow\quad \eps_k'=0 \mbox{ and } \eps_k=1 \quad \mbox{for} \quad k_0+2\leq k\leq p. \label{eq:kono-fact2}
\end{align}

Observe that by the assumption for the expression of $x$, $k_0 \to \infty$ as $h\downarrow 0$. Let $X_n(x):=1-2\eps_n(x)=(-1)^{\eps_n(x)}$ denote the $n$-th Rademacher function. For $h>0$, the following representation is a special case of Lemma 3 in \cite{Kono}:
\begin{equation*}
T(x+h)-T(x)=\Sigma_1+\Sigma_2+\Sigma_3,
\end{equation*}
where
\begin{gather*}
\Sigma_1=h\sum_{n=1}^{k_0}X_n(x)=hD_{k_0},\\
\Sigma_2=\left[\sum_{k=p+1}^\infty 2^{-k}(1-\eps_k-\eps_k')\right]\sum_{n=k_0+1}^p X_n(x), \\
\Sigma_3=\frac12\sum_{n=p+1}^{\infty}\sum_{k=n+1}^{\infty}\left[X_n(x)X_k(x)-X_n(x+h)X_k(x+h)\right]2^{-k}.
\end{gather*}
Since $\Sigma_3=O(h)$, it plays no role in determining whether $T'_+(x)=\pm \infty$. In fact, for many points $x$ the behavior of the difference quotient is controlled by $\Sigma_1$ alone, but in some cases, $\Sigma_2$ may be of the same order of magnitude but with the opposite sign. The key to the proof of Theorem \ref{thm:main}, then, is a careful analysis of this `rogue' term, especially the factor $\sum_{k=p+1}^\infty 2^{-k}(1-\eps_{k}-\eps_k')$. Note that the other factor can be written more simply: if $k_0<p$, then
\begin{equation}
\sum_{n=k_0+1}^p X_n(x)=-(p-k_0-2),
\label{eq:middle-sum}
\end{equation}
in view of \eqref{eq:kono-fact1} and \eqref{eq:kono-fact2}.

\begin{lemma} \label{lem:key-inequality}
Assume $k_0<p$. Then
\begin{equation*}
\sum_{k=p+1}^\infty 2^{-k}(1-\eps_{k}-\eps_k')\leq h.
\end{equation*}
Moreover, if $m\geq 0$ such that $\eps_{p+m+1}=0$, then
\begin{equation*}
\sum_{k=p+1}^\infty 2^{-k}(1-\eps_{k}-\eps_k')\geq -h(1-2^{-m}).
\end{equation*}
\end{lemma}

\begin{proof}
If $h<2^{-p}$, then a ``$1$" is carried from position $p+1$ to position $p$ in the addition of $x$ and $h$ because of the assumption that $k_0<p$. If $h=2^{-p}$, then $\eps_k=\eps_k'$ for all $k>p$. In both cases, we have
\begin{equation}
\sum_{k=p+1}^\infty 2^{-k}\eps_k+h=2^{-p}+\sum_{k=p+1}^\infty 2^{-k}\eps_k',
\label{eq:addition}
\end{equation}
and so
\begin{equation*}
h-\sum_{k=p+1}^\infty 2^{-k}(1-\eps_{k}-\eps_k')
=2^{-p}+\sum_{k=p+1}^\infty 2^{-k}(2\eps_k'-1)\geq 0.
\end{equation*}
For the second statement, observe that $h(1-2^{-m})\geq h-2^{-m-p}$ since $h\leq 2^{-p}$. Thus, using \eqref{eq:addition} we obtain
\begin{align*}
h(1-2^{-m})+\sum_{k=p+1}^\infty &2^{-k}(1-\eps_{k}-\eps_k')\\
&\geq 2^{-p}+\sum_{k=p+1}^\infty 2^{-k}(1-2\eps_k)-2^{-m-p}\\
&\geq 2^{-p}+\sum_{k=p+1}^\infty 2^{-k}(-1)+2\cdot 2^{-(p+m+1)}-2^{-m-p}=0,
\end{align*}
where the last inequality follows since $\eps_{p+m+1}=0$.
\end{proof}

\begin{lemma} \label{lem:maximize}
Let $c\geq 1$, and define the function $f:\{0,1,2,\dots\}\to\RR$ by
\begin{equation*}
f(m)=(1-2^{-m})(c-m).
\end{equation*}
Let $m^*$ be the largest integer $m$ where $f(m)$ attains its maximum. 
Then
\begin{equation*}
\log_2 c-2<m^*\leq \log_2 c+1.
\end{equation*}
\end{lemma}

\begin{proof}
An easy calculation gives
\begin{equation*}
f(m+1)-f(m)=2^{-(m+1)}(c+1-m)-1,
\end{equation*}
so $f(m+1)\geq f(m)$ if and only if $2^{m+1}+m\leq c+1$. Thus,
\begin{gather}
2^{m^*}+m^*-1\leq c+1, \label{eq:lower-estimate}\\
2^{m^*+1}+m^*>c+1. \label{eq:upper-estimate}
\end{gather}
From \eqref{eq:lower-estimate} we obtain
$m^*\leq\log_2 c+1$.
On the other hand, \eqref{eq:upper-estimate} yields
\begin{equation*}
2^{m^*+2}>2^{m^*+1}+m^*>c,
\end{equation*}
so that $m^*>\log_2 c-2$.
\end{proof}

\begin{proof}[Proof of Theorem \ref{thm:main}]
Since $T(1-x)=T(x)$, it is enough to prove parts (i) and (iii). Statements (ii) and (iv) follow from these by replacing $x$ with $1-x$.

We first prove part (i). Assume $D_n\to\infty$, and let $h>0$. Suppose first that $k_0\leq p-2$. 
Then $p-k_0-2\geq 0$, so it follows from \eqref{eq:middle-sum} and Lemma \ref{lem:key-inequality} that
\begin{equation*}
\Sigma_2\geq -h(p-k_0-2).
\end{equation*}
And, since $O_{k_0}=O_p-1$,
\begin{equation*}
\Sigma_1=h(O_{k_0}-I_{k_0})=h(2O_{k_0}-k_0)=h(2O_p-k_0-2),
\end{equation*}
so that
\begin{align*}
T(x+h)-T(x)&=\Sigma_1+\Sigma_2+\Sigma_3\\
&\geq h(2O_p-k_0-2)-h(p-k_0-2)+O(h)\\
&=h(2O_p-p)+O(h)=hD_p+O(h).
\end{align*}
If, on the other hand, $k_0\geq p-1$, then $\Sigma_2=O(h)$, and $O_{k_0}\geq O_p-1$. Thus,
\begin{equation*}
T(x+h)-T(x)=\Sigma_1+O(h)\geq h(D_p-2)+O(h).
\end{equation*}
In both cases,
\begin{equation*}
\frac{T(x+h)-T(x)}{h}\geq D_p+O(1),
\end{equation*}
and hence, $T_+'(x)=+\infty$.

Conversely, suppose $\liminf_{n\to\infty}D_n<\infty$. Choose a sequence $\{n_k\}$ increasing to $+\infty$ so that $\lim_{k\to\infty}D_{n_k}<\infty$, and let $p_k:=\min\{n\geq n_k:\eps_n=0\}$. 
For $h=2^{-p_k}$ and $p=p_k$, we have $k_0=p-1$ and so $\Sigma_2=O(h)$. Hence $T(x+h)-T(x)=\Sigma_1+O(h)=hD_{k_0}+O(h)$. Since $D_{k_0}=D_{p_k}-1\leq D_{n_k}$, it follows that
\begin{equation*}
\liminf_{h\downarrow 0}\frac{T(x+h)-T(x)}{h}\leq\lim_{k\to\infty}D_{n_k}+O(1)<\infty.
\end{equation*}

Next, we prove statement (iii). Suppose first that \eqref{eq:NS-condition-} holds. Let $h>0$, $2^{-p-1}<h\leq 2^{-p}$, and let $n$ be the integer such that $b_n\leq p<b_{n+1}$. Put $m=b_{n+1}-p-1$; then $p+m+1=b_{n+1}$ and so $\eps_{p+m+1}=0$, since the $b_n$'s indicate the locations of the zeros in the binary expansion of $x$. Now it follows from \eqref{eq:NS-condition-} that $b_n-2n\to\infty$, or equivalently, $D_k\to-\infty$. If $p-k_0<2$, then $\Sigma_1=hD_{k_0}$ and $\Sigma_2=O(h)$, and so
\begin{equation*}
\frac{T(x+h)-T(x)}{h}=D_{k_0}+O(1)\to-\infty.
\end{equation*}
Assume then that $p-k_0\geq 2$. By \eqref{eq:kono-fact1} and \eqref{eq:kono-fact2}, $k_0=b_n-1$, and since $O_{b_n}=n$, we have
\begin{equation*}
\Sigma_1=h(2O_{k_0}-k_0)\leq h(2O_{b_n}-k_0)=h(2n-k_0).
\end{equation*}
As for $\Sigma_2$, Lemma \ref{lem:key-inequality} gives
\begin{equation*}
\Sigma_2\leq h(1-2^{-m})(p-k_0-2)\leq h(1-2^{-m})(p-k_0).
\end{equation*}
Hence,
\begin{align}
\begin{split}
\frac{T(x+h)-T(x)}{h}&\leq 2n-k_0+(1-2^{-m})(p-k_0)+O(1)\\
&=2n-b_n+(1-2^{-m})(b_{n+1}-b_n-m)+O(1).
\end{split}
\label{eq:high-estimate}
\end{align}
For given $n$, let $m_n$ be the largest value of $m$ which maximizes the function
\begin{equation*}
f_n(m)=(1-2^{-m})(b_{n+1}-b_n-m).
\end{equation*}
By Lemma \ref{lem:maximize} we have, for any $m$,
\begin{align*}
2n-b_n+f_n(m)&\leq b_{n+1}-2b_n+2n-m_n\\
&\leq b_{n+1}-2b_n+2n-\log_2(b_{n+1}-b_n)+2.
\end{align*}
This, in combination with \eqref{eq:NS-condition-},\eqref{eq:high-estimate}, and the already established result for the case $p-k_0<2$, yields $T_+'(x)=-\infty$.

For the converse, assume that \eqref{eq:NS-condition-} fails.
Suppose first that $D_n\to-\infty$, or equivalently, $b_n-2n\to\infty$.  
Replacing the sequence $\{b_n\}$ with a suitable subsequence if necessary, we may assume there exists $M\in\RR$ such that
\begin{equation}
b_{n+1}-2b_n+2n-\log_2(b_{n+1}-b_n)>M \qquad\mbox{for all $n$}.
\label{eq:bounded-below}
\end{equation}
Fix $n\in\NN$ temporarily, let $m=m_n$, and let $h=2^{-p}$, where $p=b_{n+1}-m$. By Lemma \ref{lem:maximize},
\begin{align*}
b_{n+1}-m-b_n&\geq b_{n+1}-b_n-\log_2(b_{n+1}-b_n)-1\\
&>M+(b_n-2n)-1\to\infty.
\end{align*}
Thus, for all sufficiently large $n$, $b_n<p<b_{n+1}$. Therefore $k_0=b_n-1$, and
\begin{equation*}
\Sigma_1=h(2O_{k_0}-k_0)=h(2n-b_n-1).
\end{equation*}
Furthermore, $p-k_0-2=b_{n+1}-b_n-m-1\geq 0$ for $n$ large enough, and
\begin{align*}
\sum_{k=p+1}^\infty 2^{-k}(1-\eps_k-\eps_k')
&\leq -\sum_{k=p+1}^{b_{n+1}-1}2^{-k}+\sum_{k=b_{n+1}}^\infty 2^{-k}\\
&=-2^{-p}\left(1-2^{-(b_{n+1}-p-2)}\right)=-h\left(1-2^{-(m-2)}\right),
\end{align*}
where the inequality follows since $\eps_k=\eps_k'=1$ for $k=p+1,\dots,b_{n+1}-1$. Hence
\begin{equation*}
\Sigma_2\geq h\left(1-2^{-(m-2)}\right)(b_{n+1}-b_n-m-1).
\end{equation*}
Putting these results together, we obtain
\begin{align*}
\frac{T(x+h)-T(x)}{h}&\geq 2n-b_n-1+\left(1-2^{-(m-2)}\right)(b_{n+1}-b_n-m-1)+O(1)\\
&\geq b_{n+1}-2b_n+2n-m-2^{-(m-2)}(b_{n+1}-b_n)+O(1).
\end{align*}
By Lemma \ref{lem:maximize}, $m\leq\log_2(b_{n+1}-b_n)+1$, and the term $2^{-(m-2)}(b_{n+1}-b_n)$ is bounded. Thus,
\begin{equation}
\frac{T(x+h)-T(x)}{h}\geq b_{n+1}-2b_n+2n-\log_2(b_{n+1}-b_n)+O(1),
\label{eq:DQ-lower-bound}
\end{equation}
which is bounded below, by \eqref{eq:bounded-below}.

If $\limsup_{n\to\infty}D_n>-\infty$, then we can choose a sequence $\{n_k\}$ increasing to $+\infty$ so that $\lim_{k\to\infty}D_{n_k}>-\infty$. Let $p_k:=\max\{n\leq n_k:\eps_n=0\}$; then $D_{p_k}\geq D_{n_k}$. 
For $h=2^{-p_k}$ and $p=p_k$, we have $k_0=p-1$ and so $\Sigma_2=O(h)$. Hence $T(x+h)-T(x)=hD_{k_0}+O(h)$. Since $D_{k_0}=D_{p_k}-1\geq D_{n_k}-1$, it follows that
\begin{equation*}
\limsup_{h\downarrow 0}\frac{T(x+h)-T(x)}{h}\geq\lim_{k\to\infty}D_{n_k}+O(1)>-\infty,
\end{equation*}
completing the proof.
\end{proof}

\section{The modulus of continuity} \label{sec:modulus}

In this final section we present some exact results concerning the modulus of continuity of $T$. Let $d_1(x)$ and $d_0(x)$ denote the densities of $1$ and $0$ in the binary expansion of $x$, respectively, provided they exist (see \eqref{eq:density1} and \eqref{eq:density0}). Kruppel \cite{Kruppel} recently proved that
if $x$ is dyadic, then 
\begin{equation}
\lim_{h \to 0} \frac{T(x+h)-T(x)}{|h| \log_2(1/|h|)}=1,
\label{eq:dyadic-case}
\end{equation}
while if $x$ is non-dyadic but rational, then
\begin{equation}
\lim_{h \to 0} \frac{T(x+h)-T(x)}{h \log_2(1/|h|)}= 1-\frac{2(\eps_{k+1}+\eps_{k+2}+\cdots +\eps_{k+m})}{m},
\label{eq:modulus-limit}
\end{equation}
where $\eps_{k+1} \eps_{k+2} \cdots \eps_{k+m}$ is a repeating part in the binary expansion of $x$. Observe that for an $x$ of the latter type, $d_1(x)=(\eps_{k+1}+\eps_{k+2}+\cdots +\eps_{k+m})/m$, so the right hand side of \eqref{eq:modulus-limit} can be written as $d_0(x)-d_1(x)$.

Here we will give a simpler proof of \eqref{eq:dyadic-case}, and generalize \eqref{eq:modulus-limit} to arbitrary real numbers. More precisely, we give a complete characterization of the set of points $x$ for which the limit in \eqref{eq:modulus-limit} exists, and show that if it does, it must equal $d_0(x)-d_1(x)$. 

\begin{definition}
A point $x\in[0,1]$ is {\em density-regular} if $d_1(x)$ exists and one of the following holds:\vspace{-1mm}
\begin{enumerate}[(a)]\setlength{\itemsep}{-1mm}
\item $0<d_1(x)<1$; or
\item $d_1(x)=0$ and $a_{n+1}/a_n\to 1$; or
\item $d_1(x)=1$ and $b_{n+1}/b_n\to 1$.
\end{enumerate}
Here, $\{a_n\}$ and $\{b_n\}$ are the sequences determined by \eqref{eq:expansions}.
\end{definition}

\begin{lemma} \label{lem:right-limit}
Let $x\in[0,1]$ and suppose $d_1(x)$ exists. 

(i) If $d_1(x)<1$, then
\begin{equation}
\lim_{h \downarrow 0} \frac{T(x+h)-T(x)}{h\log_2(1/|h|)}=d_{0}(x)-d_{1}(x).
\label{eq:exact-modulus}
\end{equation} 

(ii) Suppose $d_1(x)=1$. Then
\begin{equation*}
\lim_{h \downarrow 0} \frac{T(x+h)-T(x)}{h\log_2(1/|h|)} \quad\mbox{exists}
\end{equation*} 
if and only if $b_{n+1}/b_n\to 1$, in which case the limit is equal to $-1$.
\end{lemma}

\begin{proof}
Assume throughout that $d_1(x)$ exists. If $d_1(x)<1$, then $b_{n+1}/b_n\to 1$ holds automatically (see the proof of Corollary \ref{cor:density}). Thus, we can prove the two statements by a single argument.

Suppose first that $b_{n+1}/b_n\to 1$. Then $k_0/p\to 1$ as $h\downarrow 0$. We can write 
$\Sigma_1=hD_{k_0}=h(O_{k_0}-I_{k_0})$.
Since $p \leq \log_2(1/|h|)<p+1$ and 
\begin{equation*}
\lim_{k_0 \to \infty}\frac{O_{k_0}-I_{k_0}}{k_0}=d_0(x)-d_1(x), 
\end{equation*}
it follows that 
\begin{equation*}
\lim_{h \downarrow 0} \frac{\Sigma_1}{h\log_2(1/|h|)}=d_{0}(x)-d_{1}(x).
\end{equation*} 

Next, by Lemma \ref{lem:key-inequality} and \eqref{eq:middle-sum}, we have
$|\Sigma_2|\leq h(p-k_0)$,
and hence,
\begin{equation*}
\frac{|\Sigma_2|}{h\log_2(1/|h|)}\leq \frac{p-k_0}{p}\to 0.
\end{equation*} 
Finally, since $\Sigma_3=O(h)$, \eqref{eq:exact-modulus} follows. 

Conversely, suppose $d_1(x)=1$ and $b_{n+1}/b_n$ does not tend to $1$; in other words, $\limsup_{n\to\infty}b_n/b_{n+1}<1$. On the one hand, we can choose an increasing index sequence $\{p_n\}$ such that $\eps_{p_n}=0$ for each $n$; such a sequence exists even if $x$ is dyadic, in view of our convention of choosing the representation ending in all zeros for such points. Put $h_n:=2^{-p_n}$. Then $k_0=p_n-1$, so $\Sigma_2=O(h_n)$ and
\begin{equation*}
\lim_{n\to\infty}\frac{T(x+h_n)-T(x)}{h_n\log_2(1/|h_n|)}=\lim_{n\to\infty}\frac{\Sigma_1}{h_n\log_2(1/|h_n|)}=-1,
\end{equation*}
as above. On the other hand, we can let $h$ approach $0$ along a sequence $\{h_n\}$ just as in the last part of the proof of Theorem \ref{thm:main}. Since $p=b_{n+1}-m_n\sim b_{n+1}$, dividing both sides by $\log_2(1/|h|)$ in \eqref{eq:DQ-lower-bound} gives
\begin{equation*}
\liminf_{n\to\infty}\frac{T(x+h_n)-T(x)}{h_n\log_2(1/|h_n|)}\geq \liminf_{n\to\infty}\frac{b_{n+1}-2b_n}{b_{n+1}}>-1,
\end{equation*}
since the remaining terms in \eqref{eq:DQ-lower-bound} are of smaller order than $b_{n+1}$ in view of $n/b_n\to 0$. Thus, the limit in \eqref{eq:exact-modulus} does not exist.
\end{proof}

\begin{corollary}[Kruppel \cite{Kruppel}, Proposition 3.2] \label{cor:Kruppel}
If $x$ is dyadic, then \eqref{eq:dyadic-case} holds.
\end{corollary} 
 
\begin{proof}
If $x$ is dyadic, then $d_1(x)=d_1(1-x)=0$. Thus, the statement follows by applying Lemma \ref{lem:right-limit} first to $x$ and then to $1-x$, since for $h<0$,
\begin{equation*}
\frac{T(x+h)-T(x)}{|h|\log_2(1/|h|)}
=\frac{T(1-x+|h|)-T(1-x)}{|h|\log_2(1/|h|)},
\end{equation*}
by the symmetry of $T$.
\end{proof} 

For non-dyadic $x$, we obtain the following result. Before stating it we observe that, if $\lim_{n\to\infty}n/b_n=d$, then $d_0(x)$ exists and is equal to $d$. (This is straightforward to verify.)

\begin{theorem} \label{thm:kiko}
Let $x$ be non-dyadic, and define $a_n$ and $b_n$ as in \eqref{eq:expansions}. The limit
\begin{equation*}
\lim_{h \to 0} \frac{T(x+h)-T(x)}{h \log_2(1/|h|)}
\end{equation*}
exists if and only if $x$ is density-regular, in which case the limit is equal to $d_0(x)-d_1(x)$.
\end{theorem}

\begin{proof}
If $d_1(x)$ exists, the result follows by applying Lemma \ref{lem:right-limit} first to $x$ and then to $1-x$, since $d_1(x)=1-d_1(1-x)$.

Suppose $d_1(x)$ does not exist. For $n\in\NN$, let $p=b_n$ and $h=2^{-p}$; then $k_0=p-1$, so $\Sigma_2=O(h)$ and $\Sigma_1=h D_{k_0}=h(D_p-1)=h(2n-b_n-1)$. Thus,
\begin{equation*}
\frac{T(x+h)-T(x)}{h \log_2(1/|h|)}=\frac{2n-b_n-1}{b_n}+o(1)=2\frac{n}{b_n}-1+o(1),
\end{equation*}
which does not have a limit as $n\to\infty$.
\end{proof}


\end{document}